\documentclass[dvips]{imsart}

\RequirePackage[OT1]{fontenc}
\RequirePackage{amsthm,amsmath, amssymb, enumerate}
\RequirePackage[numbers]{natbib}
\RequirePackage[colorlinks,citecolor=blue,urlcolor=blue]{hyperref}
\RequirePackage{hypernat}

\arxiv{1105.4925v1}

\startlocaldefs
\numberwithin{equation}{section}
\theoremstyle{plain}
\newtheorem{theorem}{Theorem}[section]
\newtheorem{lemma}{Lemma}[section]
\newtheorem{definition}{Definition}[section]
\newtheorem{example}{Example}[section]
\newtheorem{corollary}{Corollary}[section]

\newtheorem{remark}{Remark}[section]

\endlocaldefs

\newcommand{\R}{\mathbb R}
\newcommand{\N}{\mathbb N}

\newcommand{\Z}{\mathbb{Z}}

\begin{document}

\begin{frontmatter}
\title{A unified approach to  Stein characterizations}
\thankstext{t1}{This is an original survey paper}
\runtitle{Stein characterizations}

\begin{aug}
\author{\fnms{Christophe} \snm{Ley}\thanksref[1]{1}\ead[label=e1]{chrisley@ulb.ac.be}}
\and
\author{\fnms{Yvik} \snm{Swan}\thanksref[2]{2}\ead[label=e2]{yvswan@ulb.ac.be}}

\thankstext[1]{1}{Supported by a Mandat d'Aspirant from the Fonds National de la Recherche Scientifique, Communaut\'e fran\c{c}aise de Belgique. Christophe Ley is also member of E.C.A.R.E.S.}
\thankstext[2]{2}{Supported by a Mandat de Charg\'e de recherche from the Fonds National de la Recherche Scientifique, Communaut\'e fran\c{c}aise de Belgique.}
\runauthor{C. Ley and Y. Swan.}

\affiliation{Universit\'e Libre de Bruxelles}
\address{Department of Mathematics\\
Universit\'e Libre de Bruxelles\\
Campus Plaine -- CP210\\
B 1050 Brussels,  Belgium\\
\printead{e1},
\printead*{e2}}
%
\end{aug}

\begin{abstract}
This article deals with  Stein characterizations of probability distributions.  We  provide a general framework for interpreting these in terms of the parameters of the underlying distribution. In order to do so we introduce two concepts (a class of functions and an operator) which generalize those which were developed in the 70's by Charles Stein and  Louis Chen  for characterizing the Gaussian and the Poisson distributions. Our methodology (i) allows for writing many (if not all) known univariate Stein characterizations, (ii)  permits to  identify clearly minimal conditions under which these results hold and (iii)  provides  a straightforward tool for constructing new Stein characterizations. Our  parametric interpretation of Stein characterizations also raises a number of questions which we outline at the end of the paper. 
 \end{abstract}



\end{frontmatter}


\section{Introduction : background and motivation}

 
 {Stein's method} is a technique   for  obtaining bounds on a ``distance'' between an unknown probability distribution  and a given  {\em target distribution}. The method stems from two papers published in the 1970's by Charles Stein  (concerning Gaussian approximation, see \citep{S72})  and  Louis H. Chen (concerning Poisson approximation, see \citep{C75}). Since those days a substantial body of work has been devoted  to extensions of the method, the literature  on the subject now being vast and varied. We refer the reader to the monographs \citep{BHJ92}, \citep{BC05}, \citep{BC05b} or \citep{CGS10}.

The gist of the method can be summarized as follows.  Suppose that, for a given target distribution  $g : \mathcal{X} \to \mathcal{X}$ dominated by a measure $\mu$ on  some probability space $(\mathcal{X}, \mathcal{A}, \mu)$,  there exists a class of functions $\mathcal{F}(g) \subset \mathcal{X}^\star := \{\psi:\mathcal{X} \to \R, \mu-\mbox{measurable}\}$   and     an operator $\mathcal{T}( \cdot ,g) :\mathcal{X}^\star \to \mathcal{X}^\star$  such that 
\begin{equation}\label{eq:c1}X \sim g(\cdot) \Longleftrightarrow  {\rm E}[\mathcal{T}(f,g)(X)] = 0 \mbox{  for all  }f  \in \mathcal{F}(g), \end{equation}
where by $X \sim g(\cdot)$ we mean ${\rm P}(X\in A) = \int_Ag(y)d\mu(y)$ for all measurable $A\subset \mathcal{X}$. Now let $Z \sim g(\cdot)$ and suppose that we are interested in studying a  random object $W$ whose distribution we do not know but which we believe to be approximately  that of $Z$. After choosing a metric $d_\mathcal{H}(W, Z) := \sup_{h \in \mathcal{H}}\left|{\rm E}[h(W)]-{\rm E}[h(Z)]\right|$ for our approximation (where $\mathcal{H}$ is also a certain class of functions),
the first step in the Stein method consists in writing, for all $h \in \mathcal{H}$,
\begin{equation}\label{eq:dlEA}d_\mathcal{H}(W, Z)  = \sup_{h \in \mathcal{H}}\left|{\rm E}[\mathcal{T}(f_h,g)(W)]\right|\end{equation}
 with $f_h$ the solution of the  so-called \emph{Stein equation}
\begin{equation}\label{eq:eqst}\mathcal{T}(f_h,g)(x) = h(x)-{\rm E}[h(Z)].\end{equation}
The intuitive reason for which \eqref{eq:dlEA} is an interesting quantity to study is the following: $Z$ satisfies the rhs of~(\ref{eq:c1}), thus if the law of $W$ is close to that of $Z$, then \eqref{eq:c1} should be nearly satisfied and the rhs of \eqref{eq:dlEA} should be close to 0 for all $h$ such that $f_h \in \mathcal{H} \cap \mathcal{F}(g)$. Hence   $\left|{\rm E}[\mathcal{T}(f_h,g)(W)]\right|$ is an indicator of the $\mathcal{H}$-distance between  $W$ and $Z$. 

The secret behind the method is that not only is the intuition outlined in the previous paragraph correct, but also, as it turns out, the rhs of  \eqref{eq:dlEA} happens to be often  ``easier'' to bound, making   \eqref{eq:dlEA}  a good starting point for a wide family of stochastic approximation problems. 
Determining equations of the form   \eqref{eq:c1} for a given distribution $g$ is the crucial starting of this method.   For instance,  Stein \citep{S72}Ê  showed that \eqref{eq:c1} holds for the Gaussian with $\mathcal{T}(f, g)(x) = f'(x)-xf(x)$ and $\mathcal{F}(g)$ the class of differentiable functions on $\R$;  Chen \citep{C75} showed a similar relationship for the rate-$\lambda$ Poisson distribution, with $\mathcal{T}(f, g)(x) = f(x+1)-\lambda f(x) $ and $\mathcal{F}(g)$ the class of all bounded functions on $\Z$.
After identifying a suitable characterization, the usual methodology relies on three steps, namely (i)  solving  \eqref{eq:eqst} for all $h \in \mathcal{H}$, (ii)  deriving bounds -- the so-called \emph{magic factors} -- on the corresponding solutions, and (iii) applying the right tool (exchangeable pairs, zero- or size-biased distributions, truncation, etc.)  in order to obtain explicit bounds on the rhs of \eqref{eq:dlEA} through the bounds obtained in step (ii). We refer the reader to the recent survey \citep{R11} for an overview. 

 This method has been applied in a wide  number of problems. While the bulk of the (now vast) literature on this subject is devoted to Poisson and normal approximation problems, there have also been extensions towards other non standard densities, particularly so in recent years.      G\"otze and Tikhomirov, for instance,  use a characterization of the {semi-circular law} to obtain rates of convergence  for spectra of random matrices with martingale structure (see \citep{GT07}). Chatterjee, Fulman and R\"ollin use \emph{two} different  characterizations of the  {exponential distribution} to obtain general results for convergence towards an exponential distribution  (see \citep{CFR11}); they illustrate the applications of their methodology in a study of the spectrum of  graphs with a non-normal limit. In \citep{SDHR04}, Stein, Diaconis, Holmes and Reinert obtain a characterization  -- by means of what is now called the \emph{density approach} --  of all  {regular distributions} with a regular derivative (a function is regular if it is bounded and has at most countably many discontinuity points  on its support, see also \citep{S86} and \citep{D91});  they use this  in the analysis of simulations.  Eichelsbacher and L\"owe \citep{EL10} and Chatterjee and  Shao~\citep{CS10}   use  \citep{SDHR04}'s general characterizations to obtain general non-Gaussian approximation theorems, relevant for example in the field  of statistical mechanics.  Extensions  to a multivariate setting are also available for the multivariate Gaussian law (see, for instance,   \citep{GR96},   \citep{CM08}, or  \citep{RR09}).    There exists a uniform  treatment of the univariate discrete case, by means of Gibbs measures, which can be found in \citep{ER08}. In \citep{PRR11}Ê  a characterization of the Kummer-$U$ function is used to study degree asymptotics with rates for preferential attachment random graphs.  Extensions of the method to continuous time processes are currently the object of active research (see \citep{NPR10a},  \citep{NPR10b} or \citep{NP11} and the references therein). This list is of course not exhaustive, and the method has also been used for binomial, negative binomial, multinomial, gamma, $\chi^2$ and many other target distributions (see \citep{CGS10}). 

In this paper we will not address Stein's method \emph{per se}, but rather
concentrate on the characterizations   \eqref{eq:c1} which are known, in the literature, as \emph{Stein characterizations}. These  have, so far, never been the subject of a specific treatment in the literature, and  have always been introduced, through largely case-by-case arguments, as a means to an end rather than as an object of intrinsic interest. 
This  is perhaps explained by the nature of the different target distributions (discrete, continuous, bounded or unbounded support, etc.) which make it  complicated to try to unify all these characterizations under a single umbrella (in general different choices of target distributions will require imposing different combinations of restrictive assumptions).    The purpose of this  article  is to exploit the similarities between \emph{all} the characterizations discussed above  in order to show how \emph{all} these results can be seen as different instances of a unique phenomenon. 

As it turns out, not only does  our approach allow for (re-)obtaining the  characterizations  mentioned above  (as well as many others) but it also simplifies the resulting proofs and allows to identify clearly  the  minimal conditions on the target densities under which such characterizations hold. 
More importantly it opens new lines of research, and builds hitherto unsuspected  bridges between Stein's method and information theory. 

The outline of the paper is the following. In Section \ref{sec:heuristic} we discuss different Stein characterizations and use this discussion to provide the heuristic behind our approach. This heuristic is formalized   in Section \ref{sec:mainresult} where we also  prove our main result, Theorem \ref{theo}, which provides  a general and simple characterization theorem for a very broad family of (discrete and continuous, univariate and multivariate) distributions.  In Section \ref{sec:explicit} we illustrate the consequences of Theorem~\ref{theo} by providing  characterizations for important classes of  parametric distributions, namely the location and the scale families, as well as a characterization for discrete distributions. 
We apply our findings  in  a number of illustrative examples in Section \ref{sec:new}, and uncover a couple of unpublished  (to the best of our knowledge) characterizations.   In Section \ref{sec:furtherwork}, we discuss a couple potential applications of our results. 
Finally,  Appendix \ref{app:proofs} collects the more technical proofs.

\section{Stein characterizations}\label{sec:heuristic}

In this section we provide the heuristic behind our approach. The arguments are constructive: starting from the Gaussian distribution we generalize so as to obtain the weakest possible assumptions  for the most general result. 

\subsection{The density approach} \label{sec:heur_cont}

Let $\phi$ stand for the standard Gaussian density. In a seminal paper \citep{S72}, Charles Stein introduced the  relationship 
\begin{equation}\label{eq:SG} X\sim \mathcal{N}(0, 1) \Longleftrightarrow {\rm E}[f'(X)-Xf(X)] = 0 \mbox{ for all }Êf \in \mathcal{F}(\phi), \end{equation}
with $\mathcal{F}(\phi)$ the collection of all differentiable real  functions for which the expectation in~(\ref{eq:SG}) is defined. 
There exist many proofs of \eqref{eq:SG} (see, e.g., \citep{H78}, \citep{CS05} or \citep{NPR10a}). 
We opt to present a different argument which enjoys the advantage of being  transferable to virtually \emph{any} continuous target density. 

First  remark that $-x = \phi'(x)/\phi(x)$ so that equation \eqref{eq:SG} can be  equivalently rewritten 
\begin{align}h(x) \propto \phi(x)   
&  \Longleftrightarrow \int_\R\left(f'(x)+\frac{\phi'(x)}{\phi(x)}f(x)\right)h(x)dx = 0 \mbox{ for all }Êf \in \mathcal{F}(\phi),\nonumber \\
&  \Longleftrightarrow \int_\R\frac{\left(f\phi\right)'(x)}{ \phi(x)}  h(x)dx = 0 \mbox{ for all }Êf \in \mathcal{F}(\phi),\label{eq:SG2}
\end{align}
where $h:\R\to \R^+$ is some density.  The sufficient  condition in \eqref{eq:SG2}  is immediate via integration by parts (the implicit boundary conditions on $f\in \mathcal{F}(\phi)$ ensuring that the constant term vanishes). To prove the necessity,  choose for $A \subset \R$ a test function $f_A \in \mathcal{F}(\phi)$ that satisfies the differential equation 
\begin{equation}\label{eq:steq} \frac{\left(f_A\phi\right)'(x)}{ \phi(x)}   = \mathbb{I}_{A}(x) - \int_A \phi(y)dy\end{equation}
for   $\mathbb{I}_{A}$ the indicator of  $A$. If such a $f_A$ exists and if it belongs to $\mathcal{F}(\phi)$, then \eqref{eq:SG2} guarantees that 
$$ \int_Ah(x)dx = \int_A\phi(x)dx \mbox{ for all } A \subset \R$$  and thus   $h = \phi$. Of course  \eqref{eq:steq}   is easily solved explicitly,   yielding the candidate solution
\begin{equation}\label{eq:steqsol}f_A(x) = \frac{1}{\phi(x)}\int_{-\infty}^x \phi(z) \left(\mathbb{I}_{A}(z) -  \int_A \phi(y)dy\right)dz,\end{equation}
a function which, for all $A \subset \R$,  is readily shown to satisfy all the requirements for belonging to $\mathcal{F}(\phi)$ (see, e.g., \citep{CS05}). Hence   the result holds. 
Now note that the above argument relies nowhere on specific properties of the Gaussian $\phi$, but rather only on boundary and integrability conditions implicit in the definition of $\mathcal{F}(\phi)$ and in the solution \eqref{eq:steqsol}. It therefore suffices to replace $\phi$ by some generic density $g$ in all the above arguments and to work out conditions on $g$ so that everything runs smoothly in order to deduce, from  \eqref{eq:SG}, a general characterization theorem for continuous distributions. 

To the best of our knowledge such a general characterization result was presented for the first time  in \citep{SDHR04} under a slightly different form;  the only earlier similar attempt we have found is provided  in \citep{S01} where a construction of Stein operators  for  Pearson and Ord families of distributions is provided. Stein's \citep{SDHR04} result  is now known in the literature as the \emph{density approach}
 and  allows for recovering 
the Gaussian characterization \eqref{eq:SG}, the exponential characterization from \citep{CFR11} or  the following two examples (which are also provided in  \citep{SDHR04}).

\begin{example}\label{ex:1} Let $-\infty<a<b<\infty$.  A random variable $Z$ is $U[a, b]$ if and only if ${\rm E}[f'(Z)] = f(b^-) - f(a^+)$ for all differentiable  functions $f$.
\end{example}
\begin{example}\label{ex:exp} Let $\lambda>0$.  A random variable $Z$ is $Exp(\lambda)$ if and only if  ${\rm E}[f'(Z)-\lambda f(Z)] = -\lambda f(0^+)$ for all differentiable functions $f$.
\end{example}

The denomination \emph{density approach} is to be considered in analogy with the {\it generator approach} due to Barbour \citep{B88} and G\"otze \citep{G91}.

\subsection{Location-based and scale-based characterizations}

There   exist a number of outstanding  characterizations which \emph{cannot} be written in the form  \eqref{eq:SG2} such as \mbox{e.g.} those for discrete distributions or  for the semi-circular distribution. 
%
%
For instance,  in \citep{CFR11},  a version of Stein's method for exponential approximation is developed,  the arguments relying on two characterizations of the exponential distribution. The first, provided in   Example \ref{ex:exp},  is an instance of the density approach. The second is given by
\begin{equation}\label{eq:exp2} X\sim Exp(1)\Longleftrightarrow{\rm E}[Xf'(X)-(X-1)f(X)] = 0\end{equation}
for all  $f \in \mathcal{F}_2(Exp(1))$ a ``sufficiently large class'' of functions. 
This characterization is clearly not a consequence of the density approach. 
We nevertheless claim that \eqref{eq:exp2} stems from the same origin as the characterization in  Example~\ref{ex:exp}; to see this it is necessary to  re-interpret these results in terms of concepts inherited from a statistical point of view.

First recall how \eqref{eq:SG2} was deduced by replacing the linear term $x$ in \eqref{eq:SG} with the ratio $-g'/g$. This ratio  is a familiar object in statistics: it is the \emph{score function} $\varphi(x-\mu_0) = \left. (\partial_\mu g(x-\mu)\right|_{\mu=\mu_0})/g(x-\mu_0)$, evaluated at $\mu_0=0$, associated with the \emph{location} parameter $\mu$ of a location family $g(x-\mu)$ of distributions. Here $\partial_\mu$ stands for the derivative in the sense of distributions \mbox{w.r.t.} $\mu$. With this parametric notation in hand,  the characterization  can be rewritten as 
\begin{equation}\label{eq:loc1}
X\sim g(\cdot-\mu_0)\Longleftrightarrow{\rm E}\left[\frac{\left. \partial_\mu (f(X-\mu)g(X-\mu))\right|_{\mu=\mu_0}}{g(X-\mu_0)}\right]=0 
\end{equation}
for all $f\in\mathcal{F}(g;\mu_0)(\supset \mathcal{F}(g))$  a sufficiently large class of functions depending on both $g$ and $\mu_0$.  This shows how the \emph{density approach}  can be seen as a special instance (for $\mu_0=0$) of what we  will henceforth call the   \emph{location-based characterization}   \eqref{eq:loc1}.

Next reconsider equation \eqref{eq:exp2}.  For a given $Exp(\sigma)$  distribution, the parameter $\sigma$ is generally interpreted as a  \emph{scale} parameter.  Writing out the argument of the expectation in the rhs of (\ref{eq:loc1}) in terms of a scale parameter of a scale family $\sigma g(\sigma x)$  of distributions leads to ($\partial_\sigma$ denotes the  derivative in the sense of distributions \mbox{w.r.t.}~$\sigma$)
\begin{eqnarray*}
\frac{\partial_\sigma(f(\sigma x)\sigma g(\sigma x))\left.\right|_{\sigma=\sigma_0}}{\sigma_0 g(\sigma_0x)}&=&
xf'(\sigma_0x)+\frac{1}{\sigma_0}f(\sigma_0x)+f(\sigma_0x)x\frac{g'(\sigma_0x)}{g(\sigma_0x)}\\
&=&xf'(\sigma_0x)+f(\sigma_0x)\left(\frac{1}{\sigma_0}+x\frac{g'(\sigma_0x)}{g(\sigma_0x)}\right).
\end{eqnarray*}
For $g$ the density of an exponential distribution and $\sigma_0=1$, the latter equality corresponds to $xf'(x)+f(x)(1-x)$, which is the argument of the expectation in~(\ref{eq:exp2}) (note that for an exponential distribution, the support does not depend on the scale parameter, hence no indicator function needs to be differentiated). Thus, the second characterization of the $Exp(1)$ given in \citep{CFR11} can be viewed as a special instance (for $\sigma_0=1$) of what we  will call  a \emph{scale-based characterization} which, in its most general form, reads
\begin{equation}\label{eq:scale1}
X\sim \sigma_0g(\sigma_0\cdot)\Longleftrightarrow{\rm E}\left[\frac{\left. \partial_\sigma (f(\sigma X)\sigma g(\sigma X))\right|_{\sigma=\sigma_0}}{\sigma_0g(\sigma_0X)}\right]=0
\end{equation}
for all  $f\in\mathcal{F}(g;\sigma_0)$ a sufficiently large class of functions depending on both $g$ and $\sigma_0$.

The location- and scale-based characterizations provided above  do not, however,  cover Chen's characterization of the Poisson distribution, to cite but this well-known example.  Moreover, upon further thought,  there is no intuitive justification which would explain why only location and scale parameters should  play a special role; the tail parameter of a Student distribution or the upper and lower bounds of a uniform distribution over some interval $[a,b]$ should also be allowed to  play a crucial role in such characterizations, as well as, e.g., the parameter $\lambda$ of the Poisson distribution.
As it turns out, there exists a much neater and efficient general framework in which both the above ``general'' results turn out to be straightforward  particular cases. 

\subsection{A general characterization result}\label{sub:gcr}

In this section we fix, for simplicity,  $\mu_0=0$ and $\sigma_0=1$. Na\"ively exploiting the similarities between  \eqref{eq:loc1} and \eqref{eq:scale1}  encourages us to propose the following general conjecture.

 \
 
 \noindent {\bf Conjecture 1.} \emph{Let $g(x; \theta)$ be  a parametric family of densities with parameter~$\theta$. Suppose that $g(x; \theta)$  satisfies a number of  regularity conditions.  Fix a value $\theta_0$ of $\theta$ and denote by $\partial_\theta$ the derivative in the sense of distributions \mbox{w.r.t.} $\theta$. 
Then 
\begin{equation}\label{eq:gen1}
X \sim  g(\cdot;\theta_0)\Longleftrightarrow {\rm E}\left[ \frac{\partial_\theta(f(X;\theta)g(X;\theta))\left.\right|_{\theta=\theta_0}}{g(X;\theta_0)}\right]=0\end{equation} 
 for all $f\in\mathcal{F}(g;\theta_0)$ a sufficiently large class of functions depending on both $g$ and $\theta_0$.
}

\

This Conjecture, if true, would enjoy several advantages: all kinds of parameters $\theta$ could appear, and no difference would be made between the continuous and the discrete case. While promising, the main drawback of \eqref{eq:gen1} consists in the fact that  the conditions on the target density $g$,  as well as the structure of the family of test functions $\mathcal F(g;\theta_0)$ under which the Conjecture holds true,  remain mysterious. In order to clarify this issue, one final argument needs to be invoked, the origin of which lies, once again, in a statistical approach to such identities. 

Let  $g(x; \theta)$ be as in the Conjecture above. A classical result in likelihood theory states  that, under   regularity conditions, the expectation of the score function $ \partial_\theta g(x; \theta) / g(x; \theta)$ vanishes. The proof is very simple. Let $(\mathcal{X},m_\mathcal{X})$ be a measure space (e.g., $\R$ equipped with the Lebesgue measure or $\Z$ equipped with the counting measure). 
Since $\int_{\mathcal{X}} g(x;\theta)dm_\mathcal{X}(x)=1$, differentiating \mbox{w.r.t.} $\theta$ on both sides yields $\int_\mathcal{X} \partial_\theta g(x;\theta)dm_\mathcal{X}(x)=0$, provided that the derivative and the integral are interchangeable. This immediately shows that the expectation under $g(x;\theta)$ of $ \partial_\theta g(x; \theta) / g(x; \theta)$ equals zero. 
Now, under $g(\cdot;\theta_0)$, the rhs of (\ref{eq:gen1}) corresponds to $\int_\mathcal{X}\partial_\theta(f(x;\theta)g(x;\theta))\left.\right|_{\theta=\theta_0}dm_\mathcal{X}(x)=0$, which, under the condition of interchangeability of derivatives w.r.t. $\theta$ and integration w.r.t. $x$, can be rewritten as $\partial_\theta(\int_\mathcal{X}f(x;\theta)g(x;\theta)dm_\mathcal{X}(x))\left.\right|_{\theta=\theta_0}=0$. Thus, by analogy with the proof of the likelihood-based result, we see that, in order to belong to the class $\mathcal{F}(g;\theta_0)$, a test function $f$ should satisfy the following natural three conditions in some neighborhood $\Theta_0$ of $\theta_0$:
\begin{enumerate}[(i)]
\item there exists a real constant $c_f$ such that $\int_\mathcal{X} f(x;\theta)g(x;\theta)dm_\mathcal{X}(x)=c_f$ for all $\theta\in\Theta_0$;\vspace{0.1cm}
\item the mapping $\theta\mapsto f(x;\theta)g(x;\theta)$ is differentiable over $\Theta_0$;
\item the differentiation \mbox{w.r.t.} $\theta$ and the integral sign are interchangeable for all $\theta\in\Theta_0$.
\end{enumerate}
These conditions will be made more precise in Definition~\ref{def} of the next section. As we shall see, the first of these conditions yields the form of the candidate functions $f(x;\theta)$ (for instance $x \mapsto f(x-\theta)$ in the location case and $x\mapsto f(\theta x)$ in  the scale case) and the second and third explain the sometimes complicated conditions imposed on the test functions in the relevant literature.

   As a conclusion we stress an important fact:  nowhere in the above argument do we rely on the target density to be continuous. As we will show in the following section, the heuristic outlined above holds irrespective of the nature of the target density, and \eqref{eq:gen1} carries, as particular instances, the known characterizations for the Gaussian, the uniform, the exponential,  the semi-circular, the Poisson and the geometric, to cite but these. 

\section{Characterizations in terms of a parameter of interest}\label{sec:mainresult} In this section we present the main result of this paper, Theorem~\ref{theo}, which provides a unified framework for constructing Stein characterizations -- by means of a characterizing class of test functions and a characterizing operator -- for univariate, multivariate,  discrete and continuous distributions. As announced in the previous section, we show that all these results allow for an interpretation in terms of a {\em parameter of interest }Êof the target distribution.  
\subsection{Notations and definitions}

We first need to clearly identify the notations and vocabulary which will be used from now on. Throughout, we let 
 $k,p\in\N_0$ and consider the two measure spaces $(\mathcal{X},\mathcal{B}_{\mathcal{X}},m_{\mathcal{X}})$ and $(\Theta,\mathcal{B}_\Theta,m_\Theta)$,  where $\mathcal{X}$ is either  $\R^k$ or $\Z^k$, where $\Theta$ is a subset of $\R^p$ whose interior is non-empty, 
where $m_{\mathcal{X}}$ is either the Lebesgue measure or the counting measure, depending on the nature of $\mathcal{X}$,  where $m_\Theta$ is the Lebesgue measure, and where $\mathcal{B}_{\mathcal{X}}$ and $\mathcal{B}_\Theta$ are the corresponding $\sigma$-algebras.  In this setup we disregard the case of discrete parameter spaces  (as in, e.g., the discrete uniform); such distributions are shortly addressed in Remark \ref{rem:disc_par} at the end of the current section. 


Consider a couple $(\mathcal{X}, \Theta)$ equipped with the corresponding $\sigma$-algebras and measures. We say that the measurable function $g:\mathcal{X}\times \Theta \rightarrow\R^+$ forms a family of  $\theta$-\emph{parametric densities}, denoted by $g(\cdot; \theta)$, if $\int_\mathcal{X}g(x;\theta)dm_\mathcal{X}(x)=1$ for all $\theta \in \Theta$.  In this case we call $\theta$ the  parameter of interest for $g$.  
 When $\mathcal{X}=\R^k$, corresponding to the absolutely continuous case, the mapping $x \mapsto g(x;\theta)$ is, for all $\theta \in \Theta$, a probability density function  evaluated at the point $x\in\R^k$. When $\mathcal{X}=\Z^k$, corresponding to the discrete case, $g(x;\theta)$ is the probability mass associated with $x\in\Z^k$ and $g(\cdot ;\theta)$ therefore maps $\Z^k$ onto $[0,1]$. This unified terminology will allow us to  treat absolutely continuous and discrete distributions  in one common framework. For the sake of simplicity,  we rule out mixed distributions. 

\begin{example} $\theta$-{parametric densities} are ubiquitous in probability and statistics. Taking $g: \Z \times \R^+_0\to [0, 1]: (x, \lambda) \mapsto e^{-\lambda}\lambda^x/x!\,\mathbb{I}_\N(x)$, where $\mathbb{I}_A(\cdot)$ stands for  the indicator function of some set $A\in \mathcal{B}_{\mathcal{X}}$,  we see  the density of a   Poisson $\mathcal{P}(\lambda)$ distribution   as a $\lambda$-parametric density. Taking  $g: \R \times (\R\times \R_0^+)\to \R^+: (x,(\mu, \sigma)') \mapsto (2\pi\sigma^2)^{-1/2} e^{-(x-\mu)^2/(2\sigma^2)}$, we see the density of a  Gaussian $\mathcal{N}(\mu, \sigma)$ distribution  as a $(\mu, \sigma)$-parametric density. If, in the Gaussian case,  the scale is known (and set to $\sigma_0$),  one is then only interested in the location parameter $\mu$. Taking $\tilde{g}: \R \times \R \to \R^+: (x,\mu) \mapsto \tilde{g}(x; \mu) = g(x; (\mu, \sigma_0)')$,  we see the density of a  Gaussian $\mathcal{N}(\mu, \sigma_0)$ distribution  as a $\mu$-parametric density. Likewise, one can see the density of a uniform  $U[a, b]$ distribution as an $(a, b)$-parametric density, an $a$-parametric density or a $b$-parametric density.  In general, there are infinitely many ways to write the density of any given probability distribution  as a $\theta$-parametric density for any given $\theta$. See for instance, on this issue, the discussion on the so-called {natural parameters} of the exponential family in \citep{LC98}. 
\end{example}

Fix a couple $(\mathcal{X}, \Theta)$ as above, endowed with their respective $\sigma$-algebras and measures. Throughout this paper, the densities we shall work with all belong to the  class $\mathcal{G} := \mathcal{G}(\mathcal{X}, \Theta)$ of  $\theta$-parametric densities for which the mapping $\theta \mapsto g(\cdot ; \theta)$ is  differentiable in the sense of distributions. Such distributions may have a bounded support possibly depending on the parameter $\theta$; we will denote this support by $S_\theta:=S_\theta(g)$, be it dependent on $\theta$ or not. 

With this in hand, we are  ready to define the two fundamental concepts of this paper. These (a class of functions and an operator) mirror notions already present in the literature on Stein characterizations.  

\begin{definition}\label{def}
Let $\theta_0$ be an interior point of $\Theta$ and let $g \in \mathcal{G}$. 
We define the  class $\mathcal{F}(g;\theta_0)$ as the collection of {\em test functions} $f:\mathcal{X}\times\Theta\rightarrow\R$ such that the following three conditions are satisfied in some  neighborhood  $\Theta_0\subset \Theta$ of $\theta_0$. \vspace{0.2cm}

\noindent  Condition (i) : there exists   $c_f \in \R$   such that $\int_\mathcal{X}f(x;\theta)g(x;\theta)dm_\mathcal{X}(x)=c_f$  for all $\theta\in\Theta_0$.  \vspace{0.2cm}

\noindent Condition (ii) : the mapping $\theta \mapsto f(\cdot ; \theta)g(\cdot; \theta)$ is  differentiable in the sense of distributions over $\Theta_0$.
\vspace{0.2cm}

\noindent Condition (iii) : there exist $p$ $m_\mathcal{X}$-integrable functions $h_i: \mathcal{X} \to \R^+$, $i=1, \ldots, p$,  such that  $|\partial_{\theta_i}(f(x;\theta)g(x;\theta))|\leq h_i(x)$ over $\mathcal{X}$ for all $i=1,\ldots,p$ and for all $\theta\in\Theta_0$.

\end{definition}

The three conditions in Definition \ref{def} are to be compared to the three conditions discussed at the end of Section \ref{sub:gcr}. 

\begin{definition}\label{def2}
Let $\theta_0$ be an interior point of $\Theta$. Also let $g$ and $\mathcal{F}(g; \theta_0)$ be as above.  We define the  {\em Stein operator} $\mathcal{T}_{\theta_0}:=\mathcal{T}_{\theta_0}(\cdot, g):\mathcal{F}(g;\theta_0) \rightarrow\mathcal{X}^*$   as
\begin{equation}\label{eq:operator} \mathcal{T}_{\theta_0}(f, g)(x)=  \frac{\nabla_\theta(f(x;\theta)g(x;\theta))|_{\theta=\theta_0}}{g(x;\theta_0)}.
\end{equation}
\end{definition}

The operator defined by  \eqref{eq:operator}, inspired  by the rhs of \eqref{eq:gen1},  requires some comments. If the support of $g(\cdot; \theta)$ is  $\mathcal{X}$ itself, then the operator is obviously well-defined everywhere. If, on the contrary,  the density $g(\cdot; \theta)$ has  support $S_\theta \subset \mathcal{X}$, then there is an ambiguity which we need to avoid. To this end we adopt the convention that, whenever an expression involves the division by an indicator function $\mathbb{I}_A$ for some $A\in \mathcal{B}_{\mathcal{X}}$,  we are, in fact, multiplying the expression by the said indicator function. 
With this convention, writing out the operator in full  (whenever the gradient $\left.\nabla_\theta (f(x; \theta))\right|_{\theta=\theta_0}$ is well-defined on $\mathcal{X}$) reads
$$ \mathcal{T}_{\theta_0}(f, g)(x) = \left(\nabla_\theta (f(x; \theta)) |_{\theta=\theta_0}  + f(x; \theta_0) \dfrac{\nabla_\theta (g(x; \theta))|_{\theta=\theta_0}}{g(x; \theta_0)} \right)\mathbb{I}_{S_{\theta_0}}(x).$$
Our convention not only guarantees that the Stein operator is well-defined but also that,   for  any test function $f$,  the support  of $\mathcal{T}_{\theta_0}(f, g)(x)$ is included in  $S_{\theta_0}$. This convention was implicit throughout the discussion in the heuristic section. As already mentioned there, the usage of  derivatives in the sense of distributions of $g$ with respect to $\theta$ implies  also taking derivatives of indicator functions whenever $S_\theta$ depends on $\theta$.  

\begin{example}\label{ex:2}

(i)  Let $\mathcal{X}=\R$, $\Theta=\R$ and $g(x;\mu)=(2\pi)^{-1/2}e^{-(x-\mu)^2/2}$, the density of a univariate normal $\mathcal{N}(\mu,1)$ distribution. Clearly, $g$ belongs to $\mathcal{G}$ for all $\mu \in \R$ and its support $S_\mu=\R$ is independent of $\mu$. Fix $\mu_0=0$ and consider functions of the form $f:\R\times\R\rightarrow\R:(x,\mu)\mapsto f(x; \mu):= f_0(x-\mu)$, where $f_0:\R\rightarrow\R$ is chosen such that $f\in \mathcal{F}(g; 0)$. Restricting the operator $\mathcal{T}_{0}$ to the collection of $f$'s of this form, it becomes
$$\mathcal{T}_{0}(f, g)(x)=-f_0'(x)+xf_0(x).$$
(ii)   Let $\mathcal{X}=\Z$, $\Theta=\R_0^+$ and $g(x;\lambda)=e^{-\lambda}{\lambda ^x}/{x!}\, \mathbb{I}_{\N}(x)$, the density of a Poisson $\mathcal{P}(\lambda)$ distribution. Clearly, $g$ belongs to $\mathcal{G}$ for all $\lambda \in \R_0^+$ and its support $S_\lambda=\N$ is independent of $\lambda$. Fix $\lambda=\lambda_0$ 
 and consider functions of the form $f:\Z\times\R_0^+\rightarrow\R:(x, \lambda)\mapsto f(x; \lambda) := e^ \lambda[{\lambda}f_0(x+1)/(x+1)-f_0(x)]$, where $f_0 :\Z\rightarrow\R$ is chosen such that $f\in \mathcal{F}(g; \lambda_0)$.  Restricting the operator $\mathcal{T}_{\lambda_0}$ to the collection of $f$'s of this form, it becomes
$$\mathcal{T}_{\lambda_0}(f, g)(x)=e^{\lambda_0}\left(f_0(x+1)-\frac{x}{\lambda_0}f_0(x)\right)\mathbb{I}_{\N}(x).$$
\end{example}


Among densities $g\in \mathcal{G}$, those which satisfy the following (local) regularity assumption at a given interior point $\theta_0\in \Theta$ will play a particular role.  \\

\noindent Assumption A :  there exists a rectangular bounded neighborhood $\Theta_0\subset \Theta$ of $\theta_0$ and a $m_\mathcal{X}$-integrable function $h:\mathcal{X}\rightarrow\R^+$ such that $g(x;\theta)\leq h(x)$ over $\mathcal{X}$ for all $\theta\in\Theta_0$.\\

\noindent This assumption is weak, and is satisfied for example as soon as the target density is bounded over its support.  It does, nevertheless, exclude some well-known distributions such as, e.g., the arcsine distribution. 
\subsection{Main result}

With these notations, we are  ready to state and prove our general characterization theorem.

\begin{theorem}\label{theo} 
Let $g \in \mathcal{G}$, let $Z_{\theta}$ be distributed according to $g(\cdot; \theta)$, and  
let $X$ be a random vector taking values on $\mathcal{X}$. 
Fix an interior point $\theta_0 \in  \Theta$.   Then the following two assertions hold. 
\begin{itemize}
\item[(1)] If $X\stackrel{\mathcal{L}}{=} Z_{\theta_0}$, then ${\rm E}[\mathcal{T}_{\theta_0}(f,g)(X)]=0$ for all $f\in\mathcal{F}(g;\theta_0)$.
\item[(2)]  If $g$ also satisfies Assumption A at $\theta_0$ and if ${\rm E}[\mathcal{T}_{\theta_0}(f,g)(X)]=0$ for all $f\in\mathcal{F}(g;\theta_0)$, then \begin{equation}\label{eq:conc}X\,|\,X\in S_{\theta_0}\stackrel{\mathcal{L}}{=} Z_{\theta_0}.\end{equation}
\end{itemize}
\end{theorem}
The first statement  in Theorem \ref{theo} is standard; it implies that in order to obtain a Stein operator for a given $\theta$-parametric density $g$ at a point $\theta_0$, it suffices to find a collection of functions $f$ such that the conditions in Definition~\ref{def} hold and then apply the operator given in Definition \ref{def2}. As we will show in the next section, this  allows for recovering many well-known Stein operators, and for constructing many more. The second statement is also quite standard whenever $S_{\theta_0} = \mathcal{X}$. If $S_{\theta_0} \subset  \mathcal{X}$, then things are slightly more tricky. Indeed, in this case, equation~\eqref{eq:conc} does not imply that the law of $X$ is necessarily that of $Z _{\theta_0}$, but  rather that if the distribution of $X$  has support $S_{\theta_0}$ and if $X$ satisfies ${\rm E}[\mathcal{T}_{\theta_0}(f,g)(X)]=0$ (on $S_{\theta_0}$ by definition of $\mathcal{T}_{\theta_0}(f,g)$) for all $f\in \mathcal{F}(g; \theta_0)$, then $X$ is distributed according to $g(\cdot; \theta_0)$. This is in accordance with all other results of this form. 

\begin{proof} 

(1) Since Condition~(iii) allows for differentiating \mbox{w.r.t.} $\theta$ under the integral in Condition~(i) and since differentiating w.r.t. $\theta$ is allowed thanks to Condition~(ii), the claim follows immediately.

(2) First suppose that $p=1$, and fix $\Theta_0\subset \Theta$, a bounded (rectangular) neighborhood  of $\theta_0$ on which $g$ satisfies  Assumption A at $\theta_0$. Define, for  $A\in\mathcal{B}_\mathcal{X}$,  the mapping 
\begin{equation}\label{eq:stein_sol}
f_A:\mathcal{X}\times\Theta_0\rightarrow\R :(x,\theta)\mapsto \frac{1}{g(x;\theta)}\int_{\theta_0}^\theta l_A(x;u, \theta)g(x;u)dm_\Theta(u)
\end{equation}
with 
$$l_A(x;u, \theta):=\left(\mathbb{I}_A(x)-{\rm P}(Z_u\in A\, | \, Z_u\in S_\theta)\right)\mathbb{I}_{S_\theta}(x),$$ 
where 
$${\rm P}(Z_u\in B)=\int_\mathcal{X}\mathbb{I}_B(x)g(x;u)dm_\mathcal{X}(x)$$
 for $B \in \mathcal{B}_{\mathcal{X}}$.  Note that, for the event $[Z_u\in S_\theta]$ to have a non-zero probability, it is crucial to work in a neighborhood $\Theta_0$ rather than in $\Theta$; clearly, this event is always true when $S_\theta$ does not depend on $\theta$. To see that $f_A$ belongs to $\mathcal{F}(g;\theta_0)$, first note that 
\begin{eqnarray}
\int_\mathcal{X}f_A(x;\theta)g(x;\theta)dm_\mathcal{X}(x)&=&\int_\mathcal{X}\int_{\theta_0}^\theta l_A(x;u, \theta)g(x;u)dm_\Theta(u)dm_\mathcal{X}(x)\nonumber\\
&=&\int_{\theta_0}^\theta\int_\mathcal{X} l_A(x;u, \theta)g(x;u)dm_\mathcal{X}(x)dm_\Theta(u),\nonumber 
\end{eqnarray}
where the last equality  follows from Fubini's theorem, which can be applied for all  $\theta\in \Theta_0$, since  in this case there exists a constant $M$ such that 
$$\int_{\theta_0}^\theta\int_\mathcal{X}|l_A(x;u,\theta)|g(x;u)dm_\mathcal{X}(x)dm_\Theta(u)\leq2|\theta-\theta_0|\le M$$
for all $\theta \in \Theta_0$. We also have, by definition of $l_A$,  
\begin{align*}  & \int_\mathcal{X} l_A(x;u, \theta)g(x;u)dm_\mathcal{X}(x)\\
				& \quad \quad \quad  = {\rm P}(Z_u\in A\cap S_\theta)-{\rm P}\left(Z_u\in A\, | \, Z_u \in S_\theta\right){\rm P}(Z_u \in S_\theta)\\
				& \quad \quad \quad = 0.\end{align*}
Hence $f_A$ satisfies Condition~(i). Condition~(ii) is easily checked. Regarding Condition~(iii),  one sees  that 
\begin{equation}\label{eq:goodnews}\left.\partial_t\left(f_A(x; t)g(x; t)\right)\right|_{t=\theta} = l_A(x; \theta, \theta)g(x; \theta)+H(x;  \theta),\end{equation}
with $H(x;   \theta)$ a function whose complete expression is provided in the Appendix. As shown there, it is easy to bound $H(x;  \theta)$ uniformly in $\theta$ over $\Theta_0$ by a $m_{\mathcal X}$-integrable function. Moreover,   Assumption A  guarantees  that the same holds for $ l_A(x; \theta, \theta)g(x; \theta)$. Hence $f_A$ satisfies Condition~(iii). Wrapping up, we have thus proved that $f_A\in\mathcal{F}(g; \theta_0)$. The conclusion follows,  
since $H(x;  \theta_0) =0$ for all  $x\in\mathcal{X}$ (see the Appendix) and since, by hypothesis,  
$${\rm E}[\mathcal{T}_{\theta_0}(f_A, g)(X)]= {\rm E}[\mathbb{I}_{A\cap S_{\theta_0}}(X)-{\rm P}(Z_{\theta_0}\in A)\mathbb{I}_{S_{\theta_0}}(X)]=0.$$

Next suppose that $p>1$. Let $\theta_0:=(\theta_0^1,\ldots,\theta_0^p)$ and fix $\Theta_0:=\Theta_0^1\times\ldots\times\Theta_0^p$ a bounded (rectangular) neighborhood  of $\theta_0$ on which $g$ satisfies  Assumption~A at $\theta_0$. Define, for all $j=1,\ldots,p$ and for all $A\in \mathcal{B}_{\mathcal{X}}$, the mappings 
$$\bar\theta_0^j : \Theta_0^j \to \Theta_0: u \mapsto (\theta_0^1,\ldots,\theta_0^{j-1},u,\theta_0^{j+1},\ldots,\theta_0^p)$$
and
\begin{equation*}f_A^j:\mathcal{X}\times\Theta_0\rightarrow\R:(x, \theta) \mapsto  \frac{1}{g(x;\theta)}\int_{\theta_0^j}^{\theta^j} l_A^j(x;u, \theta^j)g(x;\bar\theta_0^j(u))dm_{\Theta^j}(u),\end{equation*}
with 
$$l_A^j(x;u, \theta^j):=\left(\mathbb{I}_A(x)-{\rm P}\left(Z_u^j\in A\, | \, Z_u^j\in S_{\bar\theta_0^j(\theta^j)}\right)\right)\mathbb{I}_{S_{\bar\theta_0^j(\theta^j)}}(x),$$
where $${\rm P}(Z_u^{j}\in B):=\int_\mathcal{X}\mathbb{I}_B(x)g(x;\bar\theta_0^j(u))dm_\mathcal{X}(x)$$ for $B\in \mathcal{B}_{\mathcal{X}}$. The $p$-variate equivalent of the function $f_A$ in (\ref{eq:stein_sol}) is given by $f_A^{(p)}(x;\theta):=\sum_{j=1}^pf_A^j(x;\theta)$. Along the same lines as for the special case $p=1$, Conditions (i)-(iii) are now easily seen to be satisfied by $f_A^{(p)}$ (we draw the reader's attention to the fact that the rectangular nature of the neighborhood $\Theta_0$ is important in order to ensure Condition (iii)). The result readily follows.
\end{proof}


\begin{remark} Nowhere in the proof  did we need to specify whether the random vector $X$ is univariate (for $k=1$) or multivariate (for $k>1$). \end{remark}

\begin{remark}\label{rem:dimp}
When $p>1$, the (vectorial) operator  $\mathcal{T}_{\theta_0}(f,g)$  contains, in a sense,  $p$ different characterizations of the $\theta=(\theta^1, \ldots, \theta^p)$-parametric density $g$ at $\theta_0$.  The  requirements (in this formulation of the result) on the   test functions $f$ are, perhaps, unnecessarily stringent. Indeed, setting $\theta^{(q)}:=(\theta^{i_1}, \ldots, \theta^{i_q})$ for $1\le i_1\le \ldots \le i_q \le p$, we can obviously  consider $g$ as a $\theta^{(q)}$-parametric density. The corresponding  $q$-dimensional sub-vector of $\mathcal{T}_{\theta_0}(f,g)$ also gives rise to a (vectorial) Stein operator 
for which the conclusions of Theorem \ref{theo} also hold at $\theta_0$, 
this  time  with a possibly  larger class of test functions $f$ (thanks to  the weakening of the requirements imposed by Condition~(iii)).   In particular, taking $q=1$, we obtain $p$ distinct one-dimensional characterizations of $g$ at $\theta_0$. This might be very helpful in approximation theorems concerning $g$.
\end{remark}

\begin{remark}
Note that both implications  in Theorem~\ref{theo} are obtained at fixed $\theta_0\in\Theta$. We attract the reader's attention to the fact that all our calculations and manipulations, as well as all the conditions on the functions at play, are consequently local around $\theta_0$. 
\end{remark}

\begin{remark} \label{rem:disc_par} All the definitions and  arguments above can be extended to encompass distributions with a discrete parameter space $\Theta$ (such as, e.g., the discrete uniform). For this it suffices, in a sense, to replace the derivatives and integrals by forward (or backward) differences and summations, respectively. Although it is easy to obtain Stein operators by this means, determining the exact conditions under which the theorem holds  nevertheless requires some care, since in this case there arise problems which originate in the interplay between the support of the target density and the parameter of interest. Because of these (structural) intricacies,  working out explicit conditions on the target density in this framework appears to be a rather sterile exercise, which is perhaps better suited to ad hoc case by case arguments. This issue will no longer be addressed within the present paper.    \end{remark}


The first statement  of  Theorem \ref{theo} can  be seen as a user-friendly  Stein operator-producing mechanism,  since  any subclass $\tilde{\mathcal{F}}(g; \theta_0)\subset \mathcal{F}(g; \theta_0)$ yields  a left-right implication, i.e. an implication of the form $$X \sim g(\cdot; \theta_0) \Longrightarrow {\rm E}\left[\mathcal{T}_{\theta_0}(f, g)(X)\right]=0 \text{ for all } f \in \tilde{\mathcal{F}}(g; \theta_0).$$ This   raises some important  questions. Indeed, consider for instance  the two operators provided in Example~\ref{ex:2}.  As it turns out, both these operators have proven to be extremely useful in applications and their properties are fundamental in the history of the Stein method. However, as already noted by a number of authors before us, they  are by no means the only such  operators for the Gaussian or the Poisson distribution; in our framework they are just  two particular instances of equation \eqref{eq:operator}    restricted  to certain very specific forms of test functions.  A natural question is therefore that of whether there exist other subclasses of test functions for which the corresponding operators would also be useful in applications. 
It is possible that this question does not allow for a  fully satisfactory answer. More precisely  it is possible that, for any given problem, there is no {\it a priori} reason why a given operator would yield better rates of convergence than any other, and perhaps in each problem a careful combination of different characterizations ({\it \`a la} Chatterjee, Fulman and R\"ollin   \citep{CFR11}) would be fruitful and would  allow for obtaining better results than those obtained by focusing on a single characterization alone. 

In any case it   seems intuitively clear that, in order for a subclass and the corresponding operator to be of practical use, they need to characterize  the law under consideration,  that is, we should have the relationship
\begin{equation*}X \sim g(\cdot; \theta_0) \Longleftrightarrow {\rm E}\left[\mathcal{T}_{\theta_0}(f, g)(X)\right]=0 \text{ for all }Êf \in \tilde{\mathcal{F}}(g; \theta_0),\end{equation*}
where the right-left implication is to be understood in the sense of (\ref{eq:conc}) in case $S_{\theta_0}$ is a strict subset of $\mathcal{X}$.  Constructing such  subclasses, which we call {\it $\theta$-characterizing for $g$ at $\theta_0$}, is relatively easy. Indeed  it suffices to adjoin the function $f_A$ defined in \eqref{eq:stein_sol} to  any collection (even empty) of test functions which satisfy the three conditions in Definition \ref{def}.   Such an approach is, however, of limited interest and, moreover, does not allow for clearly identifying the form of the corresponding operators. We therefore suggest a more constructive approach, which we describe in detail in the next section.

\section{Characterizing probability distributions}\label{sec:explicit} 

In this section we provide a general ``recipe'' which allows for constructing $\theta$-characterizing subclasses with well-identified operators.  We apply our method to build  general characterizations for location families,  scale families and discrete distributions. Many well-known Stein characterizations fall under the umbrella of these results. We also show how our method can be applied to obtain more unusual characterizations. 

\subsection{Characterizations under an exchangeability condition}  \label{sec:recipe}For the sake of simplicity, we let $k=p=1$. Fix $\theta_0 \in \Theta$ and choose $g \in \mathcal{G}$ which satisfies Assumption A at $\theta_0$. In order to construct a $\theta$-characterizing  subclass $\tilde{\mathcal{F}}(g; \theta_0)\subset \mathcal{F}(g; \theta_0)$, we suggest the following method.\\

\noindent  {\it Step 1:} Consider Condition (i) in Definition \ref{def}, which requires that we have 
 $$\int_{\mathcal{X}} f(x; \theta) g(x; \theta) dm_{\mathcal{X}}(x) = c_f$$
 for $c_f \in \R$. In many cases, the interaction between the variable $x$ and the parameter $\theta$ within the density $g$ allows to determine a favored family of test functions $\tilde{f}_0(x; \theta)$ which satisfy this condition. Moreover, these functions are usually expressible as  $\tilde{f}_0(x; \theta) = \tilde{T}(f_0; \theta)(x),$
 with $f_0 \in \mathcal{X}^\star$ and $\tilde{T}:  \mathcal{X}^\star\times \Theta \to \left(\mathcal{X}\times \Theta\right)^\star$.\\

\noindent {\it Step 2:}   For  $\tilde{T}$ and $f_0$ as given in Step 1, define the {\em exchanging operator } $T:  \mathcal{X}^\star\times \Theta \to \left(\mathcal{X}\times \Theta\right)^\star$ as a transformation which satisfies the {\em exchangeability condition} 
  \begin{equation}\label{eq:exch}\left.\partial_\theta \left(\tilde{T}(f_0; \theta)(x) \, g(x; \theta)\right) \right|_{\theta=\theta_0} =\left. \partial_y \left(T(f_0; \theta_0)(y) \, g(y; \theta_0) \right)\right|_{y=x}\end{equation}
over $\mathcal{X}$, where $\partial_y$ either means the derivative in the sense of distributions or the discrete (forward or backward) difference, and we hereby implicitly require that $T$ is such  that the derivative on the rhs of \eqref{eq:exch} is well-defined over $\mathcal{X}$. \\

\noindent  {\it Step 3:} Define the class $\mathcal{F}_0 := \mathcal{F}_0(g; \theta_0)$ as the collection of all functions $f_0 \in \mathcal{X}^\star$ such that $\tilde{T}(f_0; \theta) \in \mathcal{F}(g; \theta_0)$. Note  that we therefore  have the (new) left-right implication 
$$X \sim g(\cdot; \theta_0) \Longrightarrow {\rm E}\left[\dfrac{\left. \partial_y \left(T(f_0; \theta_0)(y) \, g(y; \theta_0) \right)\right|_{y=X}}{g(X; \theta_0)}\right]=0 \text{ for all } f_0 \in {\mathcal{F}}_0.$$
\vspace{.2cm}

\noindent  {\it Step 4:} Solve the {\em Stein equation} 
  \begin{equation}\label{eq:ch_st_eq} \left. \partial_y \left(T(f_0^A; \theta_0)(y) \, g(y; \theta_0) \right)\right|_{y=x} = l_A(x; \theta_0, \theta_0)g(x; \theta_0)\end{equation}
where $l_A(x; \theta_0, \theta_0)$ is as in the proof of Theorem \ref{theo}. If $T(\cdot; \theta_0)$ is invertible, it then suffices to check whether the corresponding  $f_0^A$ belongs to $\mathcal{F}_0$ in order to obtain the characterization 
$$X \sim g(\cdot; \theta_0) \Longleftrightarrow {\rm E}\left[\dfrac{\left. \partial_y \left(T(f_0; \theta_0)(y) \, g(y; \theta_0) \right)\right|_{y=X}}{g(X; \theta_0)}\right]=0 \text{ for all } f_0 \in {\mathcal{F}}_0,$$
where the right-left implication is to be understood, as before,  in the sense of (\ref{eq:conc}) in case the support  $S_{\theta_0}$ of $g(\cdot; \theta_0)$ is a strict subset of $\mathcal{X}$.

The resulting $\theta$-characterizing subclass $\tilde{\mathcal{F}}(g; \theta_0)$ is none other than the collection $\{\tilde{T}(f_0; \theta) \, | \, f_0 \in \mathcal{F}_0\}Ê\cup \{f_A\}$; this collection not only has the desired properties, but also is accompanied with a well-identified Stein operator. In the sequel it will be more convenient to state our results in terms of $\mathcal{F}_0$ rather than in terms of $\tilde{\mathcal F}(g; \theta_0)$. This is in accordance with all other results of this form.

There are a number  of ways in which one can extend the method presented above  to the cases $k>1$ and $p>1$. 
Also, for  given $\theta_0$ and $\theta$-parametric density $g$,  the choice of class $\mathcal{F}_0$ and  exchanging operator $T$  is not unique. 
Moreover, determining straightforward minimal conditions on the $f_0$ for the characterization to hold seems to be impossible without making further regularity assumptions on the target density $g$.  These considerations entail that it is perhaps more fruitful to tackle different  $\theta$-parametric densities with  ad hoc arguments. There are nevertheless  important instances in which one can obtain  general results with relative ease. To this end consider  the following assumption on univariate $\theta$-parametric densities. \\
 
\noindent Assumption B :     there exists $x_0\in \mathcal{X}$ such that   $$\left| \int_{\mathcal X} \left(\int_{x_0}^{x}l_A(y; \theta_0,\theta_0)g(y; \theta_0)dm_{\mathcal{X}}(y)\right)\mathbb{I}_{S_{\theta_0}}(x)dm_{\mathcal{X}}(x)\right|<\infty$$
for all $A \in \mathcal{B}_{\mathcal{X}}$, where $l_A(y; \theta_0,\theta_0)$ is defined as in the proof of Theorem \ref{theo}. \\


\noindent This  is  a  condition on the tails of the density $g(\cdot; \theta_0)$ which is, for instance, satisfied by the Gaussian  and the exponential distributions (while the latter is evident, see for the former \citep{CS05} page 4). As we will see,  Assumption B is  useful for determining general characterization results in location and scale models. 

\subsection{Location-based characterizations}\label{sub:loc}

In this subsection we apply  the method described in Section \ref{sec:recipe}  to study  laws whose parameter of interest is a location parameter.  
%
%
\begin{corollary} \label{cor:location}
Let  $k=p=1$ and $\mathcal{X}=\R = \Theta$, and   fix $\mu_0 \in \Theta$. Define $\mathcal{G}_{\text{loc}}$ as the collection of densities $g_0 : \mathcal X \to \R^+$ with support $S\subset\mathcal{X}$  such that the $\mu$-parametric density $g(x; \mu)=g_0(x-\mu)$ belongs to $\mathcal{G}$ and satisfies Assumptions A  and  B at $\mu_0$. Let     $\Theta_0 \subset  \Theta$ be as in Assumption A, and  define  $\mathcal{F}_0:=\mathcal{F}_0(g_0; \mu_0)$ as the collection of all   $f_0: \mathcal{X} \to \R$ such that  \vspace{0.2cm}

\noindent   Condition ($\mu$-i) :   $\left|\int_\mathcal{X} f_0(x) g_0(x) dm_{\mathcal{X}}(x)\right|  < \infty$,\vspace{0.2cm}

\noindent   Condition ($\mu$-ii)  : the mapping $x \mapsto  f_0(x)g_0(x)$ is  differentiable in the sense of distributions over $\mathcal{X}$,\vspace{0.2cm}

\noindent   Condition ($\mu$-iii) : there exists  a $m_{\mathcal{X}}$-integrable function $h: \mathcal{X} \to \R^+$  such that   $\left|\left.\partial_{y} (f_0(y - \mu) g_0(y-\mu)) \right|_{y=x}\right|\le h(x)$ over $\mathcal{X}$  for all $\mu\in \Theta_0$. \vspace{0.2cm}

\noindent Then  ${\mathcal{F}}_0$ is $\mu$-characterizing for $g_0$ at $\mu_0$, with $\mu$-characterizing operator
\begin{equation} \label{eq:location}\mathcal{T}_{\mu_0}(f_0, g_0) : \mathcal{X} \to \mathcal{X} : x \mapsto -\dfrac{\left.\partial_y\left(f_0(y- \mu_0)g_0(y-\mu_0)\right)\right|_{y=x}}{g_0(x-\mu_0)}.\end{equation}
\end{corollary}
A proof, which is a direct application of   the method described in Section~\ref{sec:recipe}, is provided in the Appendix.
The operator in   \eqref{eq:location}   -- as well as the conditions on the densities and the conditions on the test functions $f_0$ -- differ slightly from those already available in the literature; this matter has already been discussed in Section~\ref{sec:heuristic}.  
 
 Corollary \ref{cor:location} contains a number of well-known univariate characterizations covered in the literature.  For instance,   taking  $g(\cdot; \mu)$  to be the density of a $\mathcal{N}(\mu,1)$ (which satisfies  Assumptions A and B at $\mu_0=0$) we can use the operator provided in  Example~\ref{ex:2};  Corollary \ref{cor:location} then  leads to the famous Stein characterization of the standard normal distribution. Likewise,  introducing an artificial location parameter $\mu$ within the  exponential density with scale parameter 1 (which, again, satisfies Assumptions A and B at $\mu_0=0$)   leads to  the characterization of the exponential distribution given in Example~\ref{ex:exp}. More generally, when $g$ belongs to the (continuous) \emph{exponential family} (see \citep{H78}), one easily sees how the same manipulations allow to retrieve the known characterizations (see also  \citep{H82} or \citep{LC98}). We refer to  \citep{SDHR04}, \citep{D91}  and \citep{S01} for more location-based characterizations.

Next consider   the semi-circular law  whose   density is given by   
\begin{equation}\label{eq:scl}g_0(x-\mu)=\frac{2}{\pi\sigma^2}\sqrt{\sigma^2-(x-\mu)^2}\,\mathbb{I}_{[-\sigma,\sigma]}(x-\mu),\end{equation}
with $\mu\in\R$ being a location and $\sigma\in\R^+_0$ a known scale parameter. In the special case $\mu=0$ and $\sigma=2$, G\"otze and Tikhomirov \citep{GT07}  prove that a random variable $X$ is distributed according to \eqref{eq:scl} if and only if 
\begin{equation}\label{eq:scl_car}{\rm E}[(4-X^2)f'(X) -3X f(X)] = 0\end{equation}
for all test functions $f$ in a certain class of functions. We claim that \eqref{eq:scl_car} falls within the category of location-based characterizations. To see this it suffices to note that, although we are in a location model with target density satisfying Assumptions A and B at all points $\mu_0 \in \R$, the derivative  $g_0'(x-\mu)$  is not bounded at the edges of the support. Conditions ($\mu$-ii) and ($\mu$-iii) therefore entail some stringent requirements on the admissible class of test functions.  In order to be able to read these requirements more easily, one way to  proceed is to  consider only $f_0$'s of the form $f_0(x) = f_1(x)(\sigma^2-x^2)^r$, with $r>1/2$. Writing out the location-based characterization in terms of the functions $f_1$ instead of $f_0$ yields, for $r=1$, the expression in  \eqref{eq:scl_car};   sufficient conditions on $f_1$ for $f_0$ to belong to $\mathcal{F}_0$ are easy to provide  (see   \citep{GT07} in the case  $r=1$ and $\sigma = 2$). 

Note that, when the target density belongs to Pearson's family of distributions, there exists a general result due to \cite{S01} for obtaining Stein characterizations which encompasses many of the characterizations obtainable through Stein's density approach. We wish to stress the fact that all  these results can be recovered through our Corollary \ref{cor:location}.

And now a multivariate example. Consider a random $k$-vector $Z_{\mu_0}$ with $\mu_0 \in \R^k$ and density of the form $g(x; \mu) :=g_0(x-\mu)= g_0(x_1-\mu^1, x_2-\mu^2, \ldots, x_k-\mu^k).$ Suppose, for the sake of simplicity, that the support of $g_0(x-\mu)$  does not depend on $\mu$ (i.e. $S=\mathcal{X} = \R^k$).  One way to characterize such distributions at $\mu_0$ is to define, for  fixed $x_2, \ldots, x_k$, the univariate $\mu^1$-parametric density $g_1(x_1; \mu_1) = g_0(x_1-\mu^1, x_2-\mu^2_0, \ldots, x_k-\mu^k_0).$  Requiring that $g_1 \in \mathcal{G}$ and satisfies Assumptions A and B at $\mu_0^1$, we easily determine a class of functions $\mathcal{F}_0^1$ as in Corollary \ref{cor:location} to obtain 
\begin{align}& X \stackrel{\mathcal{L}}{=} Z_{\mu_0} \Longleftrightarrow\nonumber  \\ & {\rm E}\left[\dfrac{\partial_{y}\left.\left(f_0(y-\mu_0, X) g_0(y-\mu_0, X)\right)\right|_{y=X_1}}{g_0(X-\mu_0)}\right] = 0 \text{ for all } f_0 \in \mathcal{F}_0^1,\label{eq:loc_mult} \end{align}
where we use the abuse of notations $(y-\mu_0, X) = (y-\mu^1_0, X_2-\mu^2_0, \ldots, X_k-\mu^k_0)$ and $X-\mu_0 = (X_1-\mu^1_0, X_2-\mu^2_0, \ldots, X_k-\mu^k_0)$. The choice of $\mu^1$ as parameter of interest was of course for convenience only, and similar relationships hold for derivatives with respect to  $x_2, \ldots, x_k$ as well.  Moreover, when $Z_{\theta_0}$ has support $\mathcal{X}$ and  independent marginals, one easily sees how to aggregate these different results and write out a class of functions $\mathcal{F}_0^{(k)}$ as in Corollary~\ref{cor:location} to get 
\begin{equation}\label{eq:prout}Ê X \sim g(\cdot; \theta_0) \Longleftrightarrow  {\rm E}\left[\dfrac{\left.\nabla_y(f_0(y -\mu_0)g_0(y -\mu_0))\right|_{y=X}}{g_0(X -\mu_0)}\right]=0 \text{ for all } f_0 \in \mathcal{F}_0^{(k)}.\end{equation}

We conclude this section by showing how \eqref{eq:loc_mult} and \eqref{eq:prout} read in the Gaussian case. Here, setting  $\mu_0=0 \in \R^k$ and plugging the multivariate Gaussian density  $g(x; \mu, \Sigma)$ with $\Sigma$ a known symmetric positive definite $k\times k$ matrix into \eqref{eq:loc_mult}  we get, for $j=1, \ldots, k$,  
\begin{equation}\label{eq:multn1}X \sim \mathcal{N}(0, \Sigma) \Longleftrightarrow {\rm E} \left[\left.\partial_{y_j} (f_0(y_j, X))\right|_{y_j=X_j} -\sigma_jf_0(X)\right]=0 \text{ for all } f_0\in\mathcal{F}_0^j  \end{equation}
where we use the notations $(y_j, X) = (X_1, \ldots, X_{j-1}, y_j, X_{j+1}, \ldots, X_k)$ and   $\sigma_j:=(\Sigma^{-1}X)_j = \sum_{i=1}^k (\Sigma^{-1})_{ji}X_i.$ 
Moreover, when $\Sigma$ is the identity matrix $I_k$ we can use \eqref{eq:prout} to obtain 
\begin{equation}\label{eq:multn2}X \sim \mathcal{N}(0, I_k)  \Longleftrightarrow {\rm E}\left[\left.\nabla_y(f_0(y))\right|_{y=X}-X f_0(X)\right]=0 \text{ for all } f_0 \in \mathcal{F}_0^{(k)}.\end{equation}
These characterizations of the multivariate Gaussian are, to the best of our knowledge, new. They are to be compared with existing results given, e.g., in \citep {CM08} and \citep{RR09}.

\subsection{Scale-based characterizations} \label{sub:scale} In this subsection we apply  the  method described in Section \ref{sec:recipe}  to study  laws whose parameter of interest is a scale parameter. 

\begin{corollary} \label{cor:scale}
Let $k=p=1$, $\mathcal{X}=\R$ and $\Theta = \R_0^+$, and   fix $\sigma_0 \in \Theta$.  Define $\mathcal{G}_{\text{sca}}$ as the collection of densities $g_0 : \mathcal X \to \R^+$ with support $S\subset\mathcal{X}$  such that the $\sigma$-parametric density $g(x; \sigma)=\sigma g_0(\sigma x)$ belongs to $\mathcal{G}$ and satisfies Assumptions A  and  B at  $\sigma_0$. Let     $\Theta_0 \subset  \Theta$ be as in Assumption A, and  define  $\mathcal{F}_0:=\mathcal{F}_0(g_0; \sigma_0)$ as the collection of all  $f_0: \mathcal{X} \to \R$ such that  \vspace{0.2cm}

\noindent   Condition ($\sigma$-i) :   $\left|\int_\mathcal{X} f_0(x) g_0(x) dm_{\mathcal{X}}(x)\right|  < \infty$.\vspace{0.2cm}

\noindent   Condition ($\sigma$-ii)  : the mapping $x \mapsto x f_0(x) g_0(x)$ is  differentiable in the sense of distributions over $\mathcal{X}$,\vspace{0.2cm}

\noindent   Condition ($\sigma$-iii) : there exists a $m_{\mathcal{X}}$-integrable function $h: \mathcal{X} \to \R^+$   such that $\left|\left.\partial_{y} (y f_0(\sigma y) g_0(\sigma y)) \right|_{y=x}\right|\le h(x)$ over $\mathcal{X}$  for all $\sigma \in \Theta_0$.\vspace{0.2cm}

\noindent Then  ${\mathcal{F}}_0$ is $\sigma$-characterizing for $g_0$ at $\sigma_0$, with $\sigma$-characterizing operator
\begin{equation} \label{eq:scale}\mathcal{T}_{\sigma_0}(f_0, g_0) : \mathcal{X} \to \mathcal{X} : x \mapsto \dfrac{\partial_{y}\left(y f_0(\sigma_0 y)g_0(\sigma_0 y)\right)|_{y=x}}{\sigma_0g_0(\sigma_0 x)}.\end{equation}
\end{corollary}

The proof of Corollary \ref{cor:scale} is similar to that of Corollary \ref{cor:location}, and hence is omitted. 

As in the location case, this result can be extended in a number of ways to the multivariate setting. 
In the univariate setup, if $g$ is the exponential density with scale parameter~$\lambda$ and if $\lambda_0$ is set to 1, we retrieve the characterization \eqref{eq:exp2}. If $g_0$ is the density of a $\mathcal{N}(0,1)$ distribution, the above characterization reads 
\begin{equation}\label{eq:normscale}
X\sim\mathcal{N}(0,1) \Longleftrightarrow {\rm E}[Xf_0'(X)+(1-X^2)f_0(X)]=0
\end{equation}
for all (differentiable) $f_0\in\mathcal{F}_0$.  

\subsection{Discrete characterizations} \label{sec:discrete} Our last general result concerns   discrete distributions. In this  instance there is, in general, no unique interpretation of the parameters of interest; it depends on the law under investigation.    As will be clear from the proof of Corollary \ref{cor:discrete} below (see the Appendix), our approach in this setting allows us to dispense with Assumption B, which was needed in order to ensure Condition~(i) in Definition~\ref{def}.  However we need to strengthen Assumption A as follows. \\

\noindent Assumption A' : for $\psi(x; \theta):=\left. \partial_u\left(g(x; u)/g(0; u)\right)\right|_{u=\theta}$, there exists a neighborhood $\Theta_0$ of $\theta_0$ and a summable function $h:\Z\to\R^+$ such that 
$$\left|\Delta^+_x \left(\dfrac{\psi(x; \theta)}{\psi(x; \theta_0)}\sum_{j=0}^{x-1}l_A(j; \theta_0,\theta_0)g(j; \theta_0)\right)\right|\le h(x)$$
over $\mathcal{X}$ for all $\theta\in \Theta_0$ and for all $A \in \mathcal{B}_{\mathcal{X}}$, where $l_A(j;\theta_0, \theta_0)$ is defined as in the proof of Theorem \ref{theo} and where $\Delta^+_x$ is the forward difference with respect to $x$. \\

Assumption A' is sufficient to ensure Condition~(iii) in the discrete setting. 
It is not restrictive and is satisfied by all the (discrete) distributions we have considered. For example, in the Poisson case, the ratio ${\psi(x; \theta)}/{\psi(x; \theta_0)}$ is none other than $(\lambda/\lambda_0)^{x-1}\mathbb{I}_{\N_0}(x)$ so that known arguments (see page 65 of \citep{E05}) apply.

\begin{corollary}\label{cor:discrete} Let $k=p=1$, $\mathcal{X}=\Z$ and $\Theta \subset \R$, and   fix $\theta_0 \in \Theta$. Define $\mathcal{G}_{\text{dis}}$ as the collection of $\theta$-parametric discrete densities $g(\cdot;\theta):\mathcal{X}\to[0,1]$ with support $S\subset\mathcal{X}$, which we take of the form  $S=[N]:=\{0, \ldots, N\}$ for some $N\in \N_0\cup\{\infty\}$ not depending on $\theta$,  such that $g\in\mathcal{G}$ and satisfies Assumption A' at $\theta_0$.   Define $\mathcal{F}_0$ as the collection of all functions $f_0: \mathcal{X} \to \R$ for which there exists a summable function $h:\Z\to \R^+$ such that $\left|\Delta^+_x(f_0(x) \partial_u(g(x; u)/g(0;u))|_{u=\theta}) \right|\le h(x)$ over $\mathcal{X}$ for all $\theta \in \Theta_0$, with $\Theta_0$ as in Assumption A'. 

Then $\mathcal{F}_0$ is $\theta$-characterizing for $g$ at $\theta_0$, with $\theta$-characterizing operator 
\begin{equation*} \mathcal{T}_{\theta_0}(f_0, g)(x) =  \dfrac{ \Delta_x^+\left(f_0(x)\left.\partial_\theta\big({g(x; \theta)}/{g(0; \theta)}\big)\right|_{\theta=\theta_0}\right)}{g(x; \theta_0)}.\end{equation*}
\end{corollary}

Corollary \ref{cor:discrete} contains a number of well-known discrete characterizations covered in the literature among which, for instance, those for the Poisson  (see the operator in Example \ref{ex:2}), the geometric $Geom(p)$, with $p$-characterizing operator  $$\mathcal{T}_p(f_0, g)(x) =  -\frac{1}{p}\left((x+1) f_0(x+1)-\frac{x}{1-p}f_0(x)\right)\mathbb{I}_{\N}(x),$$ 
or the binomial $Bin(n, p)$, with $p$-characterizing operator 
$$\mathcal{T}_p(f_0, g)(x) =(1-p)^{-n-2}\left((n-x) f_0(x+1)-\frac{1-p}{p}xf_0(x)\right)\mathbb{I}_{[n]}(x).$$

The same  arguments allow, of course,  for dealing with other perhaps more exotic discrete distributions. Consider, for the sake of illustration, the case of the multinomial ${M}(n, p_1, \ldots, p_k)$, with density 
\begin{equation}\label{eq:mult_d}g(x)  = \dfrac{n!}{\prod_{j=0}^kx_j!}\prod_{j=0}^kp_j^{x_j} \mathbb{I}_{\Delta^n}(x)\end{equation}
where $x_0=n-\sum_{j=1}^kx_j$, $p_0=1-\sum_{j=1}^kp_j$ and 
$$\Delta^n = \left\{(x_1, \ldots, x_k) \in \N^k \, | \, 0 \le x_1+\ldots+x_k\le n\right\}.$$
In the same spirit as our previous multivariate characterizations, we start by transforming the problem into a univariate one. For this choose $p_1$ to be the parameter of interest, and rewrite \eqref{eq:mult_d} as 
\begin{equation*} g(x)  =\left(\binom{\bar n_1}{x_1}p_1^{x_1}(\bar p_1-p_1)^{\bar n_1-x_1} \right)     \dfrac{n!}{\bar n_1!} \dfrac{\prod_{j=2}^kp_j^{x_j}}{\prod_{j=2}^kx_j!} \mathbb{I}_{\Delta^n}(x)\end{equation*}
where, letting $\bar x_1 = \sum_{j=2}^kx_j$, we denote $\bar n_1 = n-\bar x_1$ and  $\bar p_1 = 1 - \sum_{j=2}^kp_j$.
Straightforward computations readily yield the corresponding operator 
$$\mathcal{T}_{p_1}(f_0, g)(x) =\xi(x;n)\left((\bar n_1 -x_1) f_0(x_1+1)-\dfrac{\bar p_1 - p_1}{p_1}x_1 f_0(x_1)\right)\mathbb{I}_{\Delta^n}(x),$$
with 
$$\xi(x;n) = \dfrac{\bar p_1}{(\bar p_1 - p_1)^{\bar n_1+2}}.$$

In each of the above cases, determining sufficient  conditions on the test functions $f_0$  for the operators to be $\theta$-characterizing is now a simple exercise which is left to the reader. 

\section{Uncovering new results}\label{sec:new}
 
In this final section, we tackle two examples which do not fall within the scope of the previous general results. In each case, we try to convey some intuition as to how our method works.  As will appear, each of these cases requires the development of \emph{ad hoc} arguments. 

\subsection{The uniform distribution}\label{ex:uni}
  
First take the target distribution $g$ to be the density of a uniform  $U[a, b]$ for $a\le b \in \R$, and define  $a$ to be the parameter of interest. This law is not, {\it stricto sensu}, a member of the scale family. It is, however,  easily seen that it belongs to $\mathcal{G}$ for all $a\neq b$ and satisfies Assumptions A and B at all $a<b$, with $b$ fixed. 
It is readily seen that the  exchanging operator $T(f_0; a)(x) =(x-b)/(b-a)f_0\left((x-a)/(b-a)\right)$ yields the precious relationship \eqref{eq:exch}, with  $\tilde{T}(f_0;  a) = f_0((x-a)/(b-a))$. 
%
%
This leads to the following result (the proof is left to the reader).
\begin{corollary}\label{cor:unif} Let $\mathcal{F}_0$ be the collection of all  functions $f_0:\R \to \R$ which are differentiable (in the sense of distributions) on $[0, 1]$.  Then  ${\mathcal{F}}_0$ is $a$-characterizing for $g$, with $a$-characterizing operator 
$$\mathcal{T}_a(f, g)(x) = \frac{1}{b-a}\left(\frac{x-b}{b-a} f_0'\left(\frac{x-a}{b-a}\right)+ f_0\left(\frac{x-a}{b-a}\right)\right)\mathbb{I}_{[a,b]}(x)-f_0(0)$$
for all   $f_0\in \mathcal{F}_0$.
\end{corollary}
Similarly, one can also construct a $b$-characterizing operator and a $b$-charac-terization for the uniform law on $[a, b]$. A third way to characterize this law is to proceed as in Section \ref{sub:loc}  and construct a $\mu$-characterization, for $\mu$ a   location parameter introduced by considering the density $g(x-\mu)$ and working, through Corollary \ref{cor:location}, with respect to $\mu$. This yields the expression in Example \ref{ex:1}.

\subsection{The Student distribution}
Take the target distribution $g$ to be the density of a Student $T(\nu)$ with parameter of interest $\nu\in\R^+_0$, the tail weight parameter. This law belongs to $\mathcal{G}$ for all $\nu>0$ and  satisfies Assumption A at all $\nu>0$.  
%
It is readily seen that the  exchanging operator 
$$T(f_0; \nu)(x) = -\frac{1}{2\nu}\frac{\Gamma(\nu/2)}{\Gamma((\nu+1)/2)} x\left(1+\frac{x^2}{\nu}\right)^{\nu/2}f_0\left(\frac{x^2}{\nu}\right)$$
 yields the precious relationship \eqref{eq:exch}, with
$$\tilde{T}(f_0;\nu)(x)=\frac{\Gamma(\nu/2)}{\Gamma((\nu+1)/2)}(1+x^2/\nu)^{\nu/2}f_0(x^2/\nu).$$
Sufficient conditions on $f_0$ for the now usual requirements to be fulfilled are easily imposed. This leads to the following result.
\begin{corollary}\label{cor:student} Fix $\nu>2$. Let $\mathcal{F}_0$ be the collection of differentiable (in the sense of distributions) functions $f_0: \RÊ\to \R$ such that  $|f_0(x^2)|/\sqrt{1+x^2}$ and $|xf_0'(x^2)|$ are $m_\R$-integrable.  Then $\mathcal{F}_0$ is $\nu$-characterizing for $g$, with $\nu$-characterizing operator
$$\mathcal{T}_\nu(f_0, g)(x) =\xi(x;\nu) \left(2x^2f_0'\left(\frac{x^2}{\nu}\right)-f_0\left(\frac{x^2}{\nu}\right)\left(\frac{x^2}{1+\frac{x^2}{\nu}}-\nu\right)\right),$$
where $\xi(x; \nu) = -\Gamma(\nu/2)(2\nu^2\Gamma((\nu+1)/2))^{-1}\left(1+{x^2}/{\nu}\right)^{\nu/2}$. 
\end{corollary}
The proof of this result is mainly computational and follows along the same lines as that of all other similar results provided in this paper. 



It seems appropriate to conclude on this final example. Obviously, similar parameter-based  characterizations can be obtained, by means of the same tools, for gamma,    hypergeometric, Laplace, Pareto distributions, etc. As far as we know there exists no univariate characterization which cannot be obtained through our approach. 

\section{Applications} \label{sec:furtherwork}

In all works related with Stein's method the characterization is merely the first  step in a complicated and not a little mysterious process.  In this paper we do not discuss the intricacies and subtleties of the method, and rather refer the non-initiated reader to the monographs \citep{BC05, BC05b} or \citep{CGS10} for an overview. Moreover  our parametric approach to the characterizations  has to this date never been used for  any application. The purpose of this section is to provide two simple and direct consequences of our  vision. Deeper results are still under investigation.

\subsection{Solving Stein equations}

Suppose that, for  a given parametric target distribution  $g$, we dispose of  characterizations of the form $Z \sim g(\cdot;\theta_0) \Longleftrightarrow  {\rm E}[\mathcal{T}_{\theta_0}(f,g)(Z)] = 0$ for all $f$ in $\mathcal{F}(g;\theta_0)$, where $\mathcal{T}_{\theta_0}(f,g)$ is a Stein operators. Then a \emph{Stein equation} for $g$ at $\theta_0$ is a differential equation given by 
\begin{equation}\label{eq:eqst}\mathcal{T}_{\theta_0}(f_h,g)(x) = l(x)\end{equation}
for $l: \mathcal X \to \R$ some function. 

\begin{example} 
In the Gaussian case, we obtain the location equation 
$$f'(x) -xf(x) = l(x)$$
and the scale equation 
$$xf'(x) + (1-x^2) f(x)=l(x).$$
In the Exponential case we obtain the location equation 
$$\left(f'(x) - f(x)\right)\mathbb{I}_{\R^+}(x) = l(x)$$
and the scale equation 
$$(xf'(x)-(x-1)f(x))\mathbb{I}_{\R^+}(x) = l(x).$$
In the Poisson case we obtain the $\lambda$-equation
$$ \left(f_0(x+1)-\frac{x}{\lambda_0}f_0(x)\right)\mathbb{I}_{\N}(x) = l(x).$$
\end{example}


 A careful reading of the different proofs provided in this paper shows that Theorem \ref{theo}  not only yields  Stein operators, but also solutions to the corresponding Stein equations (see  equations \eqref{eq:stein_sol}, \eqref{eq:stein_eq_loc}  and \eqref{eq:stein_eq_discrete}).  More specifically, our way of writing the operator (as a single differential) obviously allows for solving \emph{all} such equations in a unified way by simple integration. Note in particular how, in the discrete case, the solution is obtained           
through straightforward summation. In other words our approach allows for solving \emph{all} Stein equations in a routine fashion. 
 
 \subsection{Stein's method and information theory}
Although there are many consequences to our Theorem \ref{theo}, perhaps the most intuitive is that it provides a hitherto unsuspected direct link between Stein's method and information theoretic tools. Such results are, however, outside the scope and purpose of the present work and will be the subject of separate publications. We nevertheless wish to suggest the flavor of this connection, and therefore conclude the paper with a particularly appealing result. 

 Choose two parametric densities $p, q \in\mathcal{G}$ sharing the same support $S_\theta$. Take $f \in \mathcal{F}(p; \theta_0)$. We obviously have
\begin{align*}
& \mathcal{T}_{\theta_0}(f, p)(x) =  \frac {\left.\partial_\theta f(x;\theta) p(x; \theta)\right|_{\theta = \theta_0}}{p(x; \theta_0)} \\
					       & = \frac{\left.\partial_\theta f(x;\theta) q(x; \theta)\right|_{\theta = \theta_0}}{q(x; \theta_0)}\frac{p(x; \theta_0)}{p(x; \theta_0)} + \frac{f(x; \theta_0)   q(x; \theta_0)}{p(x; \theta_0)}\left.\partial_\theta \left(\frac{p(x;\theta)}{ q(x; \theta)}\right)\right|_{\theta = \theta_0}.
 \end{align*}
Straightforward simplifications then yield our final lemma. 

\begin{lemma}[Factorization of Stein operators] For all $f \in \mathcal{F}(p; \theta_0)$, we have
\begin{equation} \label{eq:char}\mathcal{T}_{\theta_0}(f, p)(x) = \mathcal{T}_{\theta_0}(f, q)(x)+ f (x; \theta_0) r_{\theta_0}(p, q)(x),\end{equation}
with 
\begin{equation} \label{eq:info}r_{\theta_0}(p, q)(x) :=    { \frac{\left.{\partial_\theta}p(x; \theta)\right|_{\theta = \theta_0}}{p(x; \theta_0)}} - { \frac{\left.{\partial_\theta}q(x; \theta)\right|_{\theta = \theta_0}}{q(x; \theta_0)}} .\end{equation}
\end{lemma}
We call the operator $r_{\theta_0}$ a \emph{generalized (standardized) score function} because  specifying the role of $\theta$ (location, scale, ...)  as well as its nature (discrete, continuous)  allows to recover a whole family of   {score functions} discussed in  \citep{J04}, \citep{KHJ05} or \citep{BJKM10}.  Such an observation obviously has an intriguing number of immediate applications, but also opens new lines of research which are currently under investigation. See \citep{LS11} for first results in this direction.

\appendix

\section{Technical proofs}\label{app:proofs}
\begin{proof}[Proof of equality \eqref{eq:goodnews}] 

First note that 
\begin{align*}  & \left.\partial_t \left(f_A(x; t) g(x; t)\right)\right|_{t=\theta} \\
		        & \quad \quad  = \left. \partial_t\left( \int_{\theta_0}^t l_A(x; u, t)g(x; u)dm_\Theta(u)\right)\right|_{t=\theta}\\
		        & \quad \quad  = l_A(x; \theta, \theta)g(x; \theta) + \int_{\theta_0}^\theta \left.\partial_t \left(l_A(x; u, t)\right)\right|_{t=\theta}g(x; u) dm_\Theta(u).\end{align*}
Now we have 
\begin{align*}	  \left.\partial_t \left(l_A(x; u, t)\right)\right|_{t=\theta}&  =   \left.\partial_t 	\left(\mathbb{I}_{S_t}(x)\right)\right|_{t=\theta} \left(\mathbb{I}_A(x)-{\rm P}(Z_u\in A\, | \, Z_u\in S_\theta)\right) \\
										       &  \quad \quad - \left.\partial_t\left({\rm P}(Z_u\in A\, | \, Z_u\in S_t)\right)\right|_{t=\theta} \mathbb{I}_{S_\theta}(x).\end{align*} 							
On the one hand, we easily see that  the function 
\begin{align*}  & H_1(x;  \theta) := \\
 &   \quad \quad  \left.\partial_t \left(\mathbb{I}_{S_t}(x)\right)\right|_{t=\theta}  \int_{\theta_0}^\theta \left(\mathbb{I}_A(x)-{\rm P}(Z_u\in A\, | \, Z_u\in S_\theta)\right) g(x; u) dm_\Theta(u)\end{align*}
is well-defined, bounded uniformly in $\theta$ over $\Theta_0$ by a $m_{\mathcal{X}}$-integrable function  and satisfies $H_1(x;  \theta_0)=0.$ On the other hand, we have 
\begin{align*}  &   \left. \partial_t\left({\rm P}(Z_u\in A\, | \, Z_u\in S_t)\right)\right|_{t=\theta}  \\
			 					  & 	\quad \quad = \dfrac{\left. \partial_t	\left( {\rm P}(Z_u\in A\cap S_t)\right) \right|_{t=\theta}}{{\rm P}(Z_u\in S_\theta)} - 	\left. \partial_t	\left( {\rm P}(Z_u\in  S_t)\right) \right|_{t=\theta}	\dfrac{{\rm P}(Z_u\in A\cap S_\theta)}{{\rm P}(Z_u\in  S_\theta)^2},\end{align*}
where  clearly both derivatives are well-defined.  Hence the function 
 $$H_2(x;  \theta):= \mathbb{I}_{S_\theta}(x) \int_{\theta_0}^\theta   \left. \partial_t\left({\rm P}(Z_u\in A\, | \, Z_u\in S_t)\right)\right|_{t=\theta}g(x; u) dm_\Theta(u)$$
is also well-defined, bounded uniformly in $\theta$ over $\Theta_0$ by a $m_{\mathcal{X}}$-integrable function  and satisfies $H_2(x;  \theta_0)=0.$
Defining 
  $$H(x;  \theta):=H_1(x;  \theta)-H_2(x;  \theta)$$
  we see that all  the assertions in the proof of Theorem \ref{theo} hold, and, moreover,  that 
\begin{align*}   \left.\partial_t \left(f_A(x; t) g(x; t)\right)\right|_{t=\theta_0}  & =  l_A(x; \theta_0, \theta_0)g(x; \theta_0)  + H(x;  \theta_0) \\ 
														& = l_A(x; \theta_0, \theta_0)g(x; \theta_0).\end{align*} 
This completes the proof of Theorem~\ref{theo}. 					       
\end{proof}

\begin{proof}[Proof of Corollary \ref{cor:location} (location)]
We apply the method described in Section \ref{sec:recipe}. \\

\noindent  {\it Step 1:} Choose $\tilde{T}(f_0; \mu)(x) = f_0(x-\mu)$. \\

\noindent  {\it Step 2:} Set $T(f_0; \mu)(x) = -f_0(x-\mu)$.\\

\noindent  {\it Step 3:}  One easily sees that, for any $f_0 \in \mathcal{F}_0$,  Conditions ($\mu$-i)-($\mu$-iii) on $f_0$ entail that Conditions (i)-(iii) are satisfied by $\tilde{T}(f_0; \mu)(x)$.\\

\noindent  {\it Step 4:}  Consider the solution of the Stein equation given by 
$$ f_0^A(x-\mu_0) = -\frac{1}{g_0(x-\mu_0)}\left( \int_{x_0}^x l_A(y; \mu_0, \mu_0)g_0(y-\mu_0) dm_\mathcal{X}(y)+c(x)\right)$$
for some $x_0\in \mathcal{X}$, where the function $x\mapsto c(x)$ has derivative (in the sense of distributions) equal to zero and is defined in such a way that $( \int_{x_0}^x l_A(y; \mu_0, \mu_0)g_0(y-\mu_0) dm_\mathcal{X}(y)+c(x))\partial_x\mathbb{I}_{S}(x-\mu_0)=0$ over $\mathcal{X}$. This function can be expressed as a sum of Dirac delta functions whose vertices are determined by $\partial_x\mathbb{I}_{S}(x-\mu_0)$.  This yields the candidate solution
\begin{equation}\label{eq:stein_eq_loc} f_0^A(x) =  -\frac{1}{g_0(x)} \left(\int_{x_0}^{x+\mu_0} l_A(y; \mu_0, \mu_0)g_0(y-\mu_0) dm_\mathcal{X}(y)+c(x+\mu_0)\right). \end{equation}
For this function to belong to $\mathcal{F}_0$, we need Condition ($\mu$-ii), which is obvious, Condition ($\mu$-iii), which is also obvious thanks to Assumption A once again, and Condition ($\mu$-i) which will hold as soon as 
\begin{align*}
&\left|\int_{\mathcal{X}}f_0^A(x)g_0(x) dm_\mathcal{X}(x)\right| \\
 &\hspace{1cm}= C+\left|\int_{\mathcal{X}} \int_{x_0}^{x+\mu_0} l_A(y; \mu_0, \mu_0)g(y-\mu_0) dm_\mathcal{X}(y)\mathbb{I}_S(x)dm_\mathcal{X}(x)\right|<\infty,
 \end{align*}
where $C=\int_{\mathcal{X}}c(x+\mu_0)\mathbb{I}_S(x)dm_\mathcal{X}(x)$ is finite. Since Assumption B then ensures that the quantity $\left|\int_{\mathcal{X}}f_0^A(x)g_0(x) dm_\mathcal{X}(x)\right|$ is bounded, Condition ($\mu$-i) is satisfied as well, which concludes the proof.
\end{proof}

\begin{proof}[Proof of Corollary \ref{cor:discrete} (discrete)]

In this framework, the exchangeability condition \eqref{eq:exch} reads 
\begin{equation}\label{eq:exch_dis} \partial_\theta (\tilde{T}(f_0; \theta)(x)\, g(x; \theta))|_{\theta=\theta_0} = \Delta^+_{x}\left(T(f_0; \theta_0)(x)g(x; \theta_0) \right),\end{equation}
for some $f_0\in\mathcal{F}_0$. In order to obtain the announced $\theta$-characterizing operator $\mathcal{T}_{\theta_0}(f_0,g)$, we define 
\begin{equation}\label{stab:dis}
\tilde{T}(f_0; \theta)(x) = \dfrac{\Delta_x^+(f_0(x) g(x; \theta))}{g(x; \theta)g(0;\theta)}
\end{equation}
and the (invertible) exchanging operator 
$$T(f_0; \theta_0)(x) = f_0(x)\dfrac{\partial_\theta (g(x; \theta)/g(0;\theta))|_{\theta=\theta_0}}{g(x; \theta_0)}.$$ 
One readily checks that these choices satisfy the exchangeability condition~(\ref{eq:exch_dis}).

Fix $\theta_0\in\Theta$. The sufficient condition is immediate. 
For the necessary condition to hold, we solve 
$$\Delta^+_{x}\left(T(f_0^A; \theta_0)(x)g(x; \theta_0) \right) =  l_A(x; \theta_0, \theta_0)g(x; \theta_0),$$
with $l_A$ as before, to obtain the candidate solution
\begin{equation}\label{eq:stein_eq_discrete}f_0^A(x) = (\psi(x; \theta_0))^{-1}\sum_{j=0}^{x-1} l_A(j; \theta_0, \theta_0)g(j; \theta_0),\end{equation}
where the sum over an empty set is 0. Assumption A' guarantees that this function  belongs to $\mathcal{F}_0$.
\end{proof}

\section*{Acknowledgements}
Christophe Ley's research is supported by aMandat de Charg\'e de recherche from the Fonds National de la Recherche Scientifique, Communaut\'e fran\c{c}aise de Belgique. Christophe Ley is also member of E.C.A.R.E.S. Yvik Swan's research is supported by a Mandat de Charg\'e de recherche from the Fonds National de la Recherche Scientifique, Communaut\'e fran\c{c}aise de Belgique.

\end{document}